\documentclass[12pt,leqno]{article}

\usepackage[cp1250]{inputenc}
\usepackage[OT4]{fontenc}
\usepackage{amsthm}
\usepackage{amsfonts}
\usepackage{amsmath}
\usepackage{amssymb}

\newcommand{\Q}{\mathbb{Q}}
\newcommand{\R}{\mathbb{R}}

\newcommand{\rank}{\operatorname{rank}}

\newtheorem{theorem} {Theorem}[section]
\newtheorem{theorem*}{Theorem}
\newtheorem{prop*} {Proposition}
\newtheorem{lemma*}{Lemma}
\newtheorem{lemma}[theorem]{Lemma}
\newtheorem{cor}[theorem]{Corollary}
\newtheorem{cor*}{Corollary}
\newtheorem{prop}[theorem] {Proposition}

\newtheorem{definition*}{Definition}
\newtheorem{rem}[theorem]{Remark}
\newtheorem{ex}[theorem]{Example}

\theoremstyle{definition}

\begin{document}

\begin{center}
{\bf   On the family of trajectories of an analytic gradient\\ flow
converging to a critical point}\\[1em]
by Zbigniew Szafraniec
\end{center}
{\bf Abstract.} Let $f:\R^n\rightarrow\R$ be an analytic function. There are presented
sufficient conditions for existence of an infinite family of trajectories 
of the gradient flow $\dot{x}=\nabla f(x)$ which converge to a critical
point of $f$.

\section{Introduction.}
Let $f:\R^n\rightarrow\R$ be an analytic function. According to {\L}ojasiewicz \cite{lojasiewicz},
the limit set of an integral curve of the dynamical system $\dot{x}=\nabla f(x)$ is either empty or contains
a single critical point of $f$. So the family of trajectories which converge to a critical point is a natural object of study
in the theory of gradient dynamical systems.

Let $f:\R^n,0\rightarrow\R,0$ be an analytic function defined in a neighborhood of the origin, having a critical point at $0$.
We shall write $T(f)$ for the set of non-trivial trajectories of the gradient flow
$\dot{x}=\nabla f(x)$ which converge to the origin. 
There is a natural problem: is $T(f)$ infinite? (In the planar case this is equivalent to the problem
whether the stable set of the origin has a non-empty interior?)

In some cases the answer is rather obvious. Let $S_r=S^{n-1}_r\cap\{f<0\}$, where $S^{n-1}_r=\{x\in\R^n\ |\ |x|=r\}$, $0<r\ll 1$.
By \cite{nowelszafraniec1}, \cite{nowelszafraniec2}, if $T(f)$ is finite then each  cohomology grup $H^i(S_r)$ is trivial for  $i\geq 1$.
Hence, if there exists $i\geq 1$ with $H^i(S_r)\neq 0$ then $T(f)$ is infinite.

In \cite{dzedzejszafraniec} (see also Section \ref{rozdzial_2}) there is presented an intrinsic  filtration of $T(f)$ given  in terms 
of characteristic exponents and asymptotic critical values of $f$. Unfortunately, these numbers are difficult
to compute. This is why in this paper we present methods which are more easy to apply.

Let $\omega:\R^n,0\rightarrow\R,0$ be the homogeneous initial form associated with $f$.
In some cases investigating properties of this form may provide simple conditions which guarantee that $T(f)$ is infinite.
Put $\Omega=S^{n-1}\cap\{\omega<0\}$. 
We shall show that $T(f)$ is infinite if at least one cohomology group $H^i(\Omega)$, where $i\geq 1$, is non-trivial.
Applying the Moussu results \cite{moussu}  one may also show that the same holds true if  there exists at least one
non-degenerate
critical point of $\omega|\Omega$ which is not a local minimum.

However, if $n=2$ and $S_r\neq S^1_r$ then none of the above assumptions would hold,
but $T(f)$ may be infinite. (See Example \ref{plik_8_3}.)

The main result of this paper says that $T(f)$ is infinite if
$\rank\, H^0(S_r)< \rank\, H^0(\Omega)$, 
that is $S_r$ has less
connected components than $\Omega$.

As a corollary we shall show that the inequality $\chi(S_r)<\chi(\Omega)$ implies that $T(f)$ is infinite.
(It is proper to add that there exist efficient methods of computing these Euler-Poincar\'e characteristics (see \cite{leckiszafraniec1}, \cite{szafraniec12}).)

In \cite{szafraniec_stab} there are presented sufficient conditions
for the stable set of the origin to have a non-empty interior which can be
applied to our problem in the case where $n\geq 3$.

It is worth pointing out that according to Moussu \cite[Theorem 3]{moussu} the family $T(f)$ always contain
trajectories which are represented by real analytic curves converging to the origin. In some cases
a family of those analytic curves can be infinite.

The paper is organized as follows. In Section \ref{rozdzial_1} we prove preliminary results about the homotopy type
of some semi-analytic sets. In Section \ref{rozdzial_2} we present properties of important geometric invariants
associated with trajectories of the gradient flow. In Section \ref{rozdzial_3} we prove the main result
of this paper (Theorem \ref{plik_8_1}), and we show how to apply it.
References \cite{bohmetal}, \cite{goldstein}, \cite{lageman}, \cite{nowelszafraniec2} present  significant related results and
applications.

\section{Preliminaries.}\label{rozdzial_1}

Let $f:\R^n,0\rightarrow\R,0$ be an analytic function defined in an open neighbourhood
of the origin. Let $\Q^+$ denote the set of positive rationals.
For $\ell\in\Q^+$, $a<0$, $y<0$ and $r>0$ we shall write
$$B_r^n=\{x\in\R^n\ |\ |x|\leq r\}\ ,\ \  S_r^{n-1}=\{x\in\R^n\ |\ |x|=r\},$$
$$V^{\ell,a}=\{x\in\R^n\setminus\{0\}\ |\ f(x)\leq a|x|^\ell\} ,\ $$
$$S_r^{\ell,a}=S_r^{n-1}\cap V^{\ell,a}\ ,\ \ B_r^{\ell,a}=B_r^n\cap V^{\ell,a},\ \ 
F^{\ell,a}(y)=f^{-1}(y)\cap V^{\ell,a},$$
$$D^{\ell,a}(y)= f^{-1}([y,0))\cap V^{\ell,a}=\{ x\in V^{\ell,a}\ |\ y\leq f(x)<0\}.$$

\begin{lemma} \label{plik_6_1} Assume that $\ell\in\Q^+$ and $a<0$.
If $0<-y\ll r\ll 1$ then
the sets $S_r^{\ell,a}$ and $F^{\ell,a}(y)$ are homotopy
equivalent. In particular, the cohomology groups $H^*(S_r^{\ell,a})$ and 
$H^*(F^{\ell,a}(y))$ are isomorphic.
\end{lemma}

\begin{proof} For $x\in  V^{\ell,a}\cup\{0\}$ lying sufficiently close to the origin we have $|x|^{1/2}\geq |f(x)|\geq |a|\cdot |x|^\ell$,
so that in particular functions $f(x)$, $|x|^2$  restricted to this set are proper.
According to the local triviality of analytic mappings (see \cite{hardt}), they are locally trivial.
So there is $r_0>0$ such that $|x|:B_{r_0}^{\ell,a}\rightarrow (0,r_0]$ is a trivial fibration.
Hence the inclusion $S_r^{\ell,a}\subset B_r^{\ell,a}$ is a homotopy equivalence
for  each $0<r\leq  r_0$.

By similar arguments, there is $y_0<0$ such that $D^{\ell,a}(y_0)\subset B_{r_0}^{\ell,a}$ and
 $f:D^{\ell,a}(y_0)\rightarrow [y_0,0)$ is
a trivial fibration, so that the inclusion $F^{\ell,a}(y)\subset D^{\ell,a}(y)$
is a homotopy equivalence for each $y_0\leq y<0$.

 So, if $0<-y\ll r\ll 1$ then we may assume that $r\leq r_0$, $y_0\leq y$, and
$$D^{\ell,a}(y)\subset B_r^{\ell,a}\subset D^{\ell,a}(y_0)\subset B_{r_0}^{\ell,a}.$$
As inclusions $D^{\ell,a}(y)\subset D^{\ell,a}(y_0)$ and $B_r^{\ell,a}\subset B_{r_0}^{\ell,a}$
are homotopy equivalencies, then $D^{\ell,a}(y)\subset B_r^{\ell,a}$ is a homotopy
equivalence too. Then $F^{\ell,a}(y)$ is homotopy equivalent to $S_r^{\ell,a}$.
\end{proof}

 For $0<-y\ll r\ll 1$ we shall write
$$ F_r(y)=B_r^n\cap f^{-1}(y)\ ,\ \ S_r=\{x\in S_r^{n-1}\ |\ f(x)<0\},$$
We call the set $F_r(y)$ the {\it real Milnor fibre}.
According to \cite{milnor}, it is either an $(n-1)$-dimensional compact manifold with boundary or an empty set.
Moreover, the sets $F_r(y)$ and $S_r$   are homotopy equivalent.

\begin{cor}\label{plik_6_2}
If $0<-y\ll r\ll 1$ then the cohomology groups $H^*(S_r)$ and $H^*(F_r(y))$ are isomorphic.
\end{cor}

Let $\omega$ be the initial form associated with $f$ and let $g=f-\omega$,
so that $f=\omega+g$. Denote by $d$ the degree of $\omega$. Hence $g=O(|x|^{d+1})$.

\begin{lemma}\label{plik_5_1} If $0<r\ll -a\ll 1$ then sets $S_r^{d,a}=S_r^{n-1}\cap\{f\leq a r^d\}$, 
$S_r^{n-1}\cap\{\omega\leq a r^d\}$ and $\Omega=S^{n-1}\cap\{\omega<0\}$
have the same homotopy type.
\end{lemma}

\begin{proof}
For $r\in\R$ sufficiently close to zero and $x\in S^{n-1}$ we have
$$f(rx)=\omega(rx)+g(rx)=r^d\omega(x)+r^{d+1}G(x,r),$$
where $G(x,r)$ is an analytic function defined in an open neighbourhood of $S^{n-1}\times\{0\}$.
Put $H(x,r)=\omega(x)+r G(x,r)$, and $H_r=H(\cdot,r):S^{n-1}\rightarrow\R$.

By \cite[Corollary 2.8]{milnor}, there exists $a_0<0$ such that any $a_0<a<0$ is a regular value of $\omega|S^{n-1}$.
Hence there exists $r_0>0$ such that $a$ is a regular value of every $H_r$, where $-r_0<r<r_0$.
Then 
$$\{(x,r)\in S^{n-1}\times (-r_0,r_0)\ |\ H(x,r)\leq a\}$$
is an $n$-dimensional manifold with boundary $S^{n-1}\times (-r_0,r_0)\cap H^{-1}(a)$.
By the implicit function theorem, the mapping $(x,r)\mapsto r$ restricted to both above manifolds
is a proper submersion. By Ehresmann's theorem, it is a locally trivial
fibration. Hence if $r$ is sufficiently close to zero then the manifolds
$ S^{n-1}\cap \{\omega\leq a\}=\{x\in S^{n-1}\ |\ H(x,0)\leq a\}$ and 
$S^{n-1}\cap \{ H_r\leq a\}=\{x\in S^{n-1}\ |\ H(x,r)\leq a\}$ are homeomorphic.

 The set $S^{n-1}\cap\{\omega\leq a\}$ is a deformation retract
of $\Omega=S^{n-1}\cap\{\omega<0\}$,
so that these sets have the same homotopy type.

We have $f(rx)=r^d H_r(x)$. Hence $x\in S^{n-1}\cap\{H_r\leq a\}$ if and only if
$rx\in S_r^{n-1}\cap\{f\leq a r^d\}$, 
and the proof is complete.
\end{proof}

\section{Geometric invariants of the function.} \label{rozdzial_2}
In the beginning of this section we present some results obtained by Kurdyka {\it et al.} \cite{kurdykaetal}, \cite{kurdykaparusinski} in the course of proving Thom's gradient conjecture. In exposition and notation we follow closely these papers.

Let $f:\R^n,0\rightarrow\R,0$ be an analytic function defined in a neighbourhood of the origin,
having a critical point at $0$. The gradient $\nabla f(x)$ splits into its radial component
$\frac{\partial f}{\partial r}(x)\frac{x}{|x|}$ and the spherical one
$\nabla'f(x)=\nabla f(x)-\frac{\partial f}{\partial r}(x)\frac{x}{|x|}$.
We shall denote $\frac{\partial f}{\partial r}$ by $\partial_rf$.

For $\epsilon>0$ define
$W^{\epsilon}=\{x\ |\ f(x)\neq 0\ ,\ \epsilon |\nabla'f|\leq |\partial_r f|\}$.
There exists a finite subset of positive rationals $L(f)\subset\Q^+$ such that for any $\epsilon>0$ and any sequence
$W^{\epsilon}\ni x\rightarrow 0$ there is a subsequence $W^{\epsilon}\ni x'\rightarrow 0$
and $\ell\in L(f)$ such that
$$\frac{|x'|\ \partial_r f(x')}{f(x')}\ \rightarrow \ \ell\ .$$
Elements of $L(f)$ are called {\it characteristic exponents}.

Fix $\ell >0$, not necessarily in $L(f)$, and consider $F=f/|x|^{\ell}$ defined in the complement of the origin.
We say that $a\in\R$ is {\it  an asymptotic critical value} of $F$ at the origin if there exists a sequence $x\rightarrow 0$, $x\neq 0$,
such that
$$  |x|\cdot \left|\nabla   F(x)\right|\ \rightarrow 0\ \ ,\ \ 
F(x)=\frac{f(x)}{|x|^\ell}\ \rightarrow\ a\ . $$
The set of asymptotic critical values of $F$ is finite.

The real number $a\neq 0$ is an asymptotic critical value if and only if there exists a sequence
$x\rightarrow 0$, $x\neq 0$, such that
$$ \frac{|\nabla'f(x)|}{|\partial_r f(x)|}\ \rightarrow 0\ \ ,\ \ 
  \frac{f(x)}{|x|^\ell}\ \rightarrow\ a\ . $$

Hence the set
$$L'(f)\ =\ \{(\ell,a)\ |\ \ell\in L(f), a < 0\ \mbox{is an asymptotic critical value of}\ f/|x|^{\ell}\}$$
is a finite subset of $\Q^+\times \R_-$, where $\R_-$ is the set of negative real numbers.

We shall write $T(f)$ for the set of non-trivial trajectories of the gradient flow
$\dot{x}=\nabla f(x)$ converging to the origin. By Section 6 of \cite{kurdykaetal}, for every such a trajectory
$x(t)$, with $x(t)\rightarrow 0$,
there exists a unique pair $(\ell',a')\in L'(f)$ such that $f(x(t))/|x(t)|^{\ell'} \rightarrow a'$.
There is a natural partition of $T(f)$ associated with $L'(f)$. Namely for $(\ell',a')\in L'(f)$,
$$T^{\ell',a'}(f)=\{ x(t)\in T(f)\ |\ f(x(t))/|x(t)|^{\ell'}\rightarrow a'\ \mbox{ as } x(t)\rightarrow 0\}.$$
In the set $\Q^+\times\R_-$ we may introduce the lexicographic order
$$(\ell',a')\leq (\ell,a)\ \ \mbox{if}\ \ \ell'<\ell, \ \mbox{or}\ \ell'=\ell\ \mbox{and}\ a'\leq a.$$
Take $(\ell,a)\in \Q^+\times\R_-\setminus L'(f)$. 
We  shall write 
$$\tilde{T}^{\ell,a}(f)=\bigcup T^{\ell',a'}(f), \mbox{ where } (\ell',a')<(\ell,a) \mbox{ and }(\ell',a')\in L'(f).$$

According to \cite{nowelszafraniec1}, there are $0<-y\ll r\ll 1$ such that each trajectory  $x(t)\in T(f)$ intersects $F_r(y)$ transversally at exactly one point.
Let $\Gamma(f)\subset F_r(y)$ be the union of all those points. So there is a natural one-to-one correspondence between trajectories in $T(f)$ and 
points in $\Gamma(f)$. The same way one may define the set $\Gamma^{\ell',a'}(f)\subset F_r(y)$  (resp. $\tilde{\Gamma}^{\ell,a}(f)\subset F_r(y)$)
whose points are in one-to-one correspondence with trajectories from $T^{\ell',a'}(f)$ (resp. $\tilde{T}^{\ell,a}(f)$).
In particular, for $(\ell,a)\in \Q^+\times\R_-\setminus L'(f)$ the set
$$\tilde{\Gamma}^{\ell,a}(f)=\bigcup \Gamma^{\ell',a'}(f), \mbox{ where } (\ell',a')<(\ell,a) \mbox{ and }(\ell',a')\in L'(f),$$
is a subset of $\Gamma(f)$.

By \cite[Theorem 12]{nowelszafraniec1}, \cite[Theorem 6]{dzedzejszafraniec} we have

\begin{theorem} \label{plik_2_1}  
If $0<-y\ll r\ll 1$ then the inclusion $\Gamma(f)\subset F_r(y)$ induces an isomorphism
$$\bar{H}^*(\Gamma(f))\ \simeq\ H^*(F_r(y)),$$
where $\bar{H}^*(\cdot)$ is the \v{C}ech-Alexander cohomology group.
In particular $\Gamma(f)$ has the same (finite) number of connected components as $F_r(y)$.

Moreover, for every $(\ell,a)\in\Q^+\times\R_-\setminus L'(f)$ the set
$\tilde\Gamma^{\ell,a}(f)$ is a closed subset of
$F^{\ell,a}(y)$. The inclusion
 induces an isomorphism
$$\bar{H}^*(\tilde\Gamma^{\ell,a}(f))\ \simeq\ H^*(F^{\ell,a}(y)).$$

\end{theorem}

\begin{rem}\label{plik_2_1_1}
If $\Gamma(f)$ is infinite then it contains at least one  compact and infinite connected component,
which is obviously not a zero-dimensional space.
If that is the case then the Menger-Urysohn dimension as well as the \v{C}ech-Lebesgue covering dimension 
of this component is at least one (see \cite{engelking}).
\end{rem}

By Lemma \ref{plik_6_1} and Corollary \ref{plik_6_2} we get

\begin{cor}\label{plik_2_3} There is an isomorphism
$\bar{H}^*(\Gamma(f))\ \simeq\ H^*(S_r)$. 
In particular $\Gamma(f)$ has the same (finite) number of connected components as $S_r$.

Moreover, for every $(\ell,a)\in\Q^+\times\R_-\setminus L'(f)$, if $0< r\ll 1$ then
$$\bar{H}^*(\tilde\Gamma^{\ell,a}(f))\ \simeq\ H^*(S_r^{\ell,a}).$$
In particular, if there exists $i\geq 1$ such that $H^i(S_r)\neq 0$ then $T(f)$ is infinite.
So, if $S_r\neq \emptyset$ and the Euler-Poincar\'e characteristic $\chi(S_r)\leq 0$,
then $T(f)$ is infinite.
\end{cor}

\begin{ex}\label{plik_2_3_1} The polynomial $f(x,y,z)=x^3+x^2z-y^2$
is weighted homogeneous. Of course $S_r\neq\emptyset$.
By \cite [p.245]{szafraniec12}, the Euler-Poincar\'e characteristic $\chi(S_r^2\cap\{f\geq 0\})=2$.
By the Alexander duality theorem we have $\chi(S_r)=0$. Hence the set $T(f)$ is infinite.
\end{ex}

\begin{prop}\label{plik_7_1} Let $d$ be the degree of the homogeneous  initial form $\omega$
associated with $f$, and let $\Omega=S^{n-1}\cap\{\omega<0\}$. If $0<-a\ll 1$ then 
$\bar{H}^*(\tilde\Gamma^{d,a}(f))\ \simeq\ H^*(\Omega)$.
  In particular,  if $H^i(\Omega)\neq 0$ for some $i\geq 1$ then $T(f)$ is infinite.
\end{prop}
\begin{proof} As $L'(f)$ is finite, if $0<-a\ll 1$ then $(d,a)\not\in L'(f)$.
By  Corollary \ref{plik_2_3}  and Lemma \ref{plik_5_1}, if $0<r\ll-a\ll 1$ then
$\bar{H}^*(\tilde\Gamma^{d,a}(f))\ \simeq H^*(S_r^{d,a})\ \simeq\ H^*(\Omega)$.

In particular, if $H^i(\Omega)\neq 0$, where $i\geq 1$, then
$\bar{H}^i(\tilde\Gamma^{d,a}(f))\ \neq\ 0$,
and so $\tilde\Gamma^{d,a}(f)$ is infinite. 
Hence $\tilde{T}^{d,a}(f)$, as well as $T(f)$, is infinite.
\end{proof}

\begin{cor}\label{plik_7_1_1} If $\omega$ is a quadratic form which may be reduced
to the diagonal form 
$-x_1^2-\cdots-x_{i+1}^2+x_{i+2}^2+\cdots+x_j^2,$
where $i\geq 1$, then
$$\bar{H}^i(\tilde\Gamma^{2,a}(f))\ \simeq\ H^i(\Omega)\ \simeq H^i(S^i)\ \neq\ 0.$$
Hence $T(f)$ is infinite.

\end{cor}

\begin{ex} Let $f(x,y,z)=z(x^2+y^2)+x^2y^2 z-z^4$. It is easy to see that $S_r=S_r^2\cap\{f<0\}$
is homeomorphic to a union of two disjoint 2-discs, so that $H^i(S_r)=0$ for $i\geq 1$.
As $\omega=z(x^2+y^2)$, then $\Omega$ is homeomorphic to $S^1\times(0,1)$,
and so $H^1(\Omega)\neq 0$. Hence $T(f)$ is infinite.
\end{ex}

\begin{cor}\label{plik_7_2} 
If    $\Omega\neq \emptyset$  and  the  Euler-Poincar\'e characteristic  $\chi(\Omega)\leq 0$,  then $T(f)$ is infinite.
\end{cor}

Investigating the gradient flow in polar coordinates and applying arguments presented by Moussu in \cite[p.449]{moussu}
the reader may also prove the next proposition. (As its proof would require to introduce other techniques ,
so we omit it here.)

\begin{prop}\label{plik_7_3}
Suppose that  there exists a non-degenerate critical point of $\omega|\Omega$
which is not a local minimum. Then $T(f)$ is infinite. 

In particular, if there exists a non-degenerate local maximum of $\omega|\Omega$
then the interior of the stable set of the origin is non-empty.
\end{prop}

\begin{ex} Let $f(x,y)=x^3+3xy^2+x^2 y^2$, so that $\omega=x^3+3xy^2$. 
It is easy to see that $\omega|S^1$ has a non-degenerate  local maximum at $(-1,0)\in\Omega$.
Then the interior of the stable set of the origin is non-empty.
In particular $T(f)$ is infinite.
\end{ex}

\section{Main results}\label{rozdzial_3}

\begin{theorem}\label{plik_8_1}
Suppose that $f:\R^n,0\rightarrow\R,0$ is an analytic function having
a critical point at the origin

If $\rank\,H^0(S_r)<\rank\, H^0(\Omega)$, i.e. the number of connected components of $S_r$ 
is smaller than the number of connected components of $\Omega$, then
the set of trajectories of the gradient flow $\dot{x}=\nabla f(x)$ converging to the origin
is infinite.
\end{theorem}
\begin{proof}
Suppoose, contrary to our claim, that $T(f)$ if finite. Then $\Gamma(f)$ is finite, and
 for any $(\ell,a)\in \Q^+\times\R_-\setminus L'(f)$ the set $\tilde\Gamma^{\ell,a}(f)$ is finite too. Hence 
$\rank \bar{H}^0(\tilde\Gamma^{\ell,a}(f))$
equals the number of elements in $\tilde\Gamma^{\ell,a}(f)$.

By Lemma \ref{plik_5_1}, there exist $0<r\ll -a\ll 1$ such that $\Omega$ and  $S_r^{d,a}$ have the same homotopy type.
By Corollary \ref{plik_2_3}, the group $H^*(S_r)$ is isomorphic to $\bar{H}^*(\Gamma(f))$.
Hence $\rank H^0(S_r)$$=\rank \bar{H}^0(\Gamma(f))$
equals the number of elements in $\Gamma(f)$. Moreover,
$\rank H^0(\Omega)$$=\rank H^0(S_r^{d,a})$$=\rank \bar{H}^0(\tilde\Gamma^{d,a}(f))$
equals the number of elements in $\tilde\Gamma^{d,a}(f)$.

As $\tilde\Gamma^{d,a}(f)\subset \Gamma(f)$, then $\rank H^0(\Omega)\leq\rank H^0(S_r)$,
which contradicts the assumption.
\end{proof}

\begin{theorem}\label{plik_8_2}
If $\chi(S_r)<\chi(\Omega)$ then $T(f)$ is infinite.
\end{theorem}

\begin{proof} By Corollary \ref{plik_2_3} and Proposition \ref{plik_7_1}, it is enough to consider the case
where all cohomology groups $H^i(S_r)$, $H^i(\Omega)$, where $i\geq 1$, are trivial.

If that is the case then
$\rank H^0(S_r)=\chi(S_r)<\chi(\Omega)=\rank H^0(\Omega).$
By Theorem \ref{plik_8_1}, the set $T(f)$ is infinite.
\end{proof}

\begin{ex}\label{plik_8_3} Let $f(x,y)=x^3-y^2$, so that $\omega=-y^2$.
Then $\Omega=\{ (x,y)\in S^1\ |\ -y^2<0    \}= S^1\setminus \{(\pm 1, 0)\}$.
Obviously $\Omega$ has two connected components and $H^i(\Omega)=0$ for any $i\geq 1$.
The function $\omega |\Omega$ has exactly two critical (minimum) points at $(0,\pm 1)$.
As $S_r$ is homeomorphic to an interval, then by Theorem \ref{plik_8_1} the set $T(f)$ is infinite.
\end{ex}

\begin{ex}\label{plik_8_4} Let $f(x,y,z)=xyz-z^4$, so that $\omega=xyz$.
It is easy to see that $\Omega$ is homeomorphic to a disjoint union of four discs,
and $S_r$ is homeomorphic to a disjoint union of two discs. By Theorem \ref{plik_8_1} the set
$T(f)$ is infinite.
\end{ex}

\begin{ex}\label{plik_8_5} Let $f(x,y,z)=xyz+x^4y-2y^4z+3xz^4$,
so that $f$ has an isolated critical point at the origin and $\omega=xyz$.
Applying Andrzej {\L}\c ecki computer program (see \cite{leckiszafraniec1}) we have verified that
the local topological degree of the mapping
$$\R^3,0\,\ni (x,y,z)\ \mapsto\ -\nabla f(x,y,z)\, \in\R^3,0$$
equals zero. By \cite{khimshiashvili1}, \cite{khimshiashvili2}, the Euler-Poincar\'e characteristic $\chi(S_r^2\cap\{f\geq 0\})=1-0=1$.
By the Alexander duality theorem $\chi(S_r)=1$. By Theorem \ref{plik_8_2} the set $T(f)$ is infinite.
\end{ex}

Zbigniew SZAFRANIEC\\
Institute of Mathematics, University of Gda\'nsk\\
80-952 Gda\'nsk, Wita Stwosza 57, Poland\\
Zbigniew.Szafraniec@mat.ug.edu.pl\\

%%%%%%%%%%%%%%%%%%%%%%%%%%%%%%%%%%%%%%%%%%%%%%%%%%%%%%%%%%%%%%%%%%%%%%%%%%%%%%%%%%%%%%%%%%%%%%%
\end{document}